\documentclass[11pt,oneside,english]{amsart}
\usepackage[T1]{fontenc}
\usepackage[latin9]{inputenc}
\usepackage{geometry}
\usepackage{amsthm}
\usepackage{amssymb}
\usepackage[all]{xy}

\makeatletter
\numberwithin{equation}{section}
\numberwithin{figure}{section}
\theoremstyle{plain}
\newtheorem{thm}{\protect\theoremname}
  \theoremstyle{definition}
  \newtheorem{example}[thm]{\protect\examplename}
\newenvironment{lyxcode}
{\par\begin{list}{}{
\setlength{\rightmargin}{\leftmargin}
\setlength{\listparindent}{0pt}
\raggedright
\setlength{\itemsep}{0pt}
\setlength{\parsep}{0pt}
\normalfont\ttfamily}%
 \item[]}
{\end{list}}
  \theoremstyle{plain}
  \newtheorem{prop}[thm]{\protect\propositionname}
  \theoremstyle{remark}
  \newtheorem{rem}[thm]{\protect\remarkname}
  \theoremstyle{plain}
  \newtheorem{lem}[thm]{\protect\lemmaname}
  \theoremstyle{plain}
  \newtheorem{cor}[thm]{\protect\corollaryname}

\usepackage{multicol}
\usepackage{enumerate}
\usepackage{hyperref}
\hypersetup{
    colorlinks=true,       
    linkcolor=blue,          
    citecolor=blue,        
    filecolor=blue,      
    urlcolor=blue           
}

\makeatother

\usepackage{babel}
  \providecommand{\corollaryname}{Corollary}
  \providecommand{\examplename}{Example}
  \providecommand{\lemmaname}{Lemma}
  \providecommand{\propositionname}{Proposition}
  \providecommand{\remarkname}{Remark}
\providecommand{\theoremname}{Theorem}

\begin{document}
\author{Alvaro Liendo}
\address{Instituto de Matem\'atica y F\'\i sica, Universidad de Talca, Casilla 721, Talca, Chile}
\email{aliendo@inst-mat.utalca.cl}

\author{Charlie Petitjean}   
\address{Instituto de Matem\'atica y F\'\i sica, Universidad de Talca, Casilla 721, Talca, Chile}
\email{petitjean.charlie@gmail.com}


\thanks{This research was partially supported by projects Fondecyt regular
1160864 and Fondecyt postdoctorado 3160005.}

\title{smooth varieties with torus actions}
\begin{abstract}
In this paper we provide a characterization of smooth algebraic varieties
endowed with a faithful algebraic torus action in terms of a combinatorial
description given by Altmann and Hausen. Our main result is that such
a variety $X$ is smooth if and only if it is locally isomorphic in
the étale topology to the affine space endowed with a linear torus
action. Furthermore, this is the case if and only if the combinatorial
data describing $X$ is locally isomorphic in the étale topology to
the combinatorial data describing affine space endowed with a linear
torus action. Finally, we provide an effective method to check the
smoothness of a $\mathbb{G}_{m}$-threefold in terms of the combinatorial
data. 
\end{abstract}

\maketitle

\section*{Introduction}

Let $\mathrm{k}$ be an algebraically closed field of characteristic
zero and let $\mathbb{T}$ be the algebraic torus $\mathbb{T}=(\mathbb{G}_{m})^{k}$
of dimension $k$ where $\mathbb{G}_{m}$ is the multiplicative group
of the field $(\mathrm{k^{*}},\cdot)$. The variety $\mathbb{T}$
has a natural structure of algebraic group. We denote by $M$ its
character lattice and by $N$ its 1-parameter subgroup lattice. In
this paper a variety denotes a integral separated scheme of finite
type. A $\mathbb{T}$-variety is a normal variety $X$ endowed with
a faithful action of $\mathbb{T}$ acting on $X$ by regular automorphisms.
The assumption that the $\mathbb{T}$-action on $X$ is faithful is
not a restriction since given any regular $\mathbb{T}$-action $\alpha$,
the kernel $\ker\alpha$ is a normal algebraic subgroup of $\mathbb{T}$
and $\mathbb{T}/(\ker\alpha)$ is again an algebraic torus acting
faithfully on $X$.

The complexity of a $\mathbb{T}$-variety is the codimension of the
generic orbit. Furthermore, since the action is assumed to be faithful,
the complexity of $X$ is given by $\dim X-\dim\mathbb{T}$. The best
known example of $\mathbb{T}$-varieties are those of complexity zero,
i.e., toric varieties. Toric varieties were first introduced by Demazure
in \cite{D2} as a tool to study subgroups of the Creomna group. Toric
varieties allow a combinatorial description in term of certain collections
of strongly convex polyhedral cones in the vector space $N_{\mathbb{Q}}=N\otimes_{\mathbb{Z}}\mathbb{Q}$
called fans, see for instance \cite{Od88,Fu93,C-L-S}. 

For higher complexity there is also a combinatorial description of
a $\mathbb{T}$-variety. We will use the language of p-divisors first
introduced by Altmann and Hausen in \cite{A-H} for the affine case
and generalized in \cite{A-H-S} to arbitrary $\mathbb{T}$-varieties.
This decription, that we call the A-H combinatorial description, generalizes
previously known partial cases such as \cite{KKMSD,D ,F-Z,Ti08}.
See \cite{A-I-P-S-V} for a detailed survey on the topic. 

Since the introduction of the A-H description, a lot of work has been
done in generalizing the known results from toric geometry to the
more general case of $\mathbb{T}$-varieties. On of the most basic
parts of the theory that are still open is the characterization of
smooth $\mathbb{T}$-varieties of complexity higher than one. In particular,
several clasification of singularities of $\mathbb{T}$-varieties
are given in \cite{Li-S}, but no smoothness criterion is given in
complexity higher than one. In this paper we achieve such a characterization
in arbitrary complexity. 

By Sumihiro's Theorem \cite{Su74}, every (normal) $\mathbb{T}$-variety
admits an affine cover by $\mathbb{T}$-invariant affine open sets.
Hence, to study smoothness, it is enough to consider affine $\mathbb{T}$-varieties.
The A-H description of an affine $\mathbb{T}$-variety $X$ consists
in a couple $(Y,\mathcal{D})$, where $Y$ is a normal semiprojective
variety that is a kind of quotient of $X$, usually called the Chow
quotient and $\mathcal{D}$ is a divisor on $Y$ whose coefficients
are not integers as usual, but polyhedra in $N_{\mathbb{Q}}$, see
Section~\ref{sec:AH} for details. 

It is well known that an affine toric variety $X$ of dimension $n$
is smooth if and only if it is equivariantly isomorphic to a $\mathbb{T}$-invariant
open set in $\mathbb{A}^{n}$ endowed with the standard $\mathbb{T}$-action
of complexity 0 by component-wise multiplication. Our main results
states that an affine $\mathbb{T}$-variety $X$ of dimension $n$
and of arbitrary complexity is smooth if and only if it is locally
isomorphic in the étale topology to the affine space $\mathbb{A}^{n}$
endowed with a linear $\mathbb{T}$-action, see Proposition~\ref{prop:luna-AH}.
Furthermore, $X$ is smooth if and only if the combinatorial data
$(Y,\mathcal{D})$ is locally isomorphic in the étale topology to
the combinatorial data $(Y',\mathcal{D}')$ of the affine space $\mathbb{A}^{n}$
endowed with a linear $\mathbb{T}$-action, see Theorem~\ref{Thm: smoothness}.
The main ingredient in our result is Luna's Slice Theorem \cite{L}. 

In order to effectively apply Theorem~\ref{Thm: smoothness}, it
is necessary to know the A-H description of $\mathbb{A}^{n}$ endowed
with all possible linear $\mathbb{T}$-action. The case of complexity
0 and 1 are well known, see Corollary~\ref{cor:cp1}. In Proposition~\ref{classification A3  },
we compute the combinatorial description of all the $\mathbb{G}_{m}$-action
on $\mathbb{A}^{3}$. We apply this proposition to give several examples
illustrating the behavior of the combinatorial data of smooth and
singular affine $\mathbb{G}_{m}$-varieties of complexity 2. 

We sate our results in the case of an algebraically closed field of
characteristic zero since the A-H description is given in that case.
Nevertheless, all our arguments are characteristic free. In the introduction
of \cite{A-H} it is stated that they expect their arguments to hold
in positive characteristic with the same proofs. Hence, our results
are also valid in positive characteristic provided the A-H description
is valid too. In particular, see \cite{Al10} for Luna's Slice Theorem
in positive characteristic and \cite{La15} for a different proof
of the A-H description in complexity one over arbitrary fields.

\section{Altmann-Hausen presentation of $\mathbb{T}$-varieties\label{sec:AH}}

In this section, we present a combinatorial description of affine
$\mathbb{T}$-varieties. We use here the description first introduced
by Altmann and Hausen in \cite{A-H}. This description, that we call
the A-H description, generalizes that of toric varieties \cite{C-L-S}
and many particular cases that were treared before, see \cite{KKMSD,D ,F-Z,Ti08}.
Furthermore, this was later generalized to the non affine case in
\cite{A-H-S}. See \cite{A-I-P-S-V} for a detailed survey on the
subject. 

Let $N\simeq\mathbb{Z}^{k}$ be a lattice of rank $k$, and let $M=\mathrm{Hom}(N,\mathbb{Z})$
be its dual lattice. We let $\mathbb{T}=\mathrm{Spec}(\mathrm{k}[M])$
be the algebraic torus whose character lattice is $M$ and whose 1-parameter
subgroup lattice is $N$. To both these lattices we associate rational
vector spaces $N_{\mathbb{Q}}:=N\otimes_{\mathbb{Z}}\mathbb{Q}$ and
$M_{\mathbb{Q}}:=M\otimes_{\mathbb{Z}}\mathbb{Q}$, respectively.
In all this paper, we let $\sigma\subseteq N_{\mathbb{Q}}$ be a strongly
convex polyhedral cone, i.e., the intersection of finitely many closed
linear half spaces in $N_{\mathbb{Q}}$ which does not contain any
line. Let $\sigma^{\vee}\subseteq M_{\mathbb{Q}}$ be its dual cone
and let $\sigma_{M}^{\vee}:=\sigma^{\vee}\cap M$ be the semi-group
of lattice point in $M$ contained in $\sigma^{\vee}$. Let $\Delta\subseteq N_{\mathbb{Q}}$
be a convex polyhedron, i.e., the intersection of finitely many closed
affine half spaces in $N_{\mathbb{Q}}$. Then $\Delta$ admits a decomposition
using the Minkowski sum as 
\[
\Delta:=\Pi+\sigma=\{v_{1}+v_{2}\mid v_{1}\in\Pi,v_{2}\in\sigma\},
\]
 where $\Pi$ is a polytope, i.e., the convex hull of finitely many
points in $N_{\mathbb{Q}}$ and $\sigma$ is a strongly convex polyhedral
cone. In this decomposition, the cone $\sigma$ is uniquely determined
by $\Delta$. The cone $\sigma$ is called the tail cone of $\Delta$
and the polyhedron $\Delta$ is said to be a $\sigma$-polyhedron.
Let $\mathrm{Pol_{\sigma}}(N_{\mathbb{Q}})$ be the set of all $\sigma$-polyhedra
in $N_{\mathbb{Q}}$. The set $\mathrm{Pol_{\sigma}}(N_{\mathbb{Q}})$
with Minkowski sum as operation forms a commutative semi-group with
neutral element $\sigma$. 

Let $Y$ be a semiprojective variety, i.e., $\varGamma(Y,\mathcal{O}_{Y})$
is finitely generated and $Y$ is projective over $Y_{0}=\mathrm{Spec(\varGamma(Y,\mathcal{O}_{Y}))}$.
A polyhedral divisor $\mathcal{D}$ on $Y$ is a formal sum $\mathcal{D}=\sum\Delta_{i}\cdot D_{i}$,
where $D_{i}$ are prime divisors on $Y$ and $\Delta_{i}$ are $\sigma$-polyhedron
with $\Delta_{i}=\sigma$ except for finitely many of them.

Let $\mathcal{D}$ be a polyhedral divisor on $Y$. We define the
evaluation divisor $\mathcal{D}(u)$ for every $u\in\sigma_{M}^{\vee}$
as the Weil $\mathbb{Q}$-divisor $Y$ given by 
\[
\mathcal{D}(u)=\sum\underset{v\in\Delta_{i}}{\mathrm{min}}\left\langle u,v\right\rangle D_{i}\quad\mbox{for all}\quad u\in\sigma_{M}^{\vee}.
\]

A polyhedral divisor on on $Y$ is called a p-divisor if the evaluation
divisor $\mathcal{D}(u)$ is semi-ample for each $u\in\sigma_{M}^{\vee}$
and big for each $u\in\mathrm{relint}(\sigma^{\vee})\cap M$. For
every p-divisor $\mathcal{D}$ on $Y$ we can associate a sheaf of
$\mathcal{O}_{Y}$-algebras 
\[
\mathcal{A}(Y,\mathcal{D})=\bigoplus_{u\in\sigma_{M}^{\vee}}\mathcal{O}_{Y}(\mathcal{D}(u))\cdot\chi^{u},
\]
 and its ring of global sections 
\[
A(Y,\mathcal{D})=\varGamma(Y,\mathcal{A}(Y,\mathcal{D}))=\bigoplus_{u\in\sigma_{M}^{\vee}}A_{u},\quad\mbox{where}\quad A_{u}=\varGamma(Y,\mathcal{O}_{Y}(\mathcal{D}(u)))\cdot\chi^{u}.
\]
To the $M$-graded algebra $A(Y,\mathcal{D})$, we associate the scheme
$X(Y,\mathcal{D})=\mathrm{Spec}(A(Y,\mathcal{D})).$ The main result
by Altmann and Hausen in \cite{A-H} is the following.
\begin{thm}
[Altmann-Hausen]For any p-divisor $\mathcal{D}$ on a normal semi-projective
variety $Y$, the scheme $X(Y,\mathcal{D})$ is a normal affine $\mathbb{T}$-variety
of dimension $\mathrm{dim}(Y)+\mathrm{dim}(\mathbb{T})$. Conversely
any normal affine $\mathbb{T}$-variety is isomorphic to an $X(Y,\mathcal{D})$
for some semi-projective variety $Y$ and some p-divisor $\mathcal{D}$
on $Y$.
\end{thm}
The semi-projective variety $Y$ serving as base for the combinatorial
data in the above theorem is not unique. Indeed, in the following
example we give three different presentations for the same affine
$\mathbb{T}$-variety. Computations can be carried out following the
method described in \cite[Section~11]{A-H}
\begin{example}
Let $\mathbb{A}^{3}=\mathrm{Spec}(\mathrm{k}[x,y,z])$ endowed with
the linear $\mathbb{G}_{m}$-action given by $\lambda\cdot(x,y,z)\rightarrow(\lambda x,\lambda^{-1}y,\lambda z)$.
We say that the weight matrix is $\begin{pmatrix}1 & -1 & 1\end{pmatrix}^{t}$.
In this case, the lattices $M$ and $N$ are isomorphic to $\mathbb{Z}$
and the tail cone must be $\sigma=\left\{ 0\right\} $. then $\mathbb{A}^{3}$
is equivariantly isomorphic to\end{example}
\begin{enumerate}
\item [$(i)$] $X(Y_{1},\mathcal{D}_{1})$, where $Y_{1}$ is the is the
blow-up of $\mathbb{A}^{2}$ at the origin and $\mathcal{D}_{1}=\left[-1,0\right]\cdot E$,
with $E$ the exceptional divisor of the blow up.
\item [$(ii)$] $X(Y_{2},\mathcal{D}_{2})$, where $Y_{2}$ is the is the
blow-up of $\mathbb{A}^{2}$ at the origin and $\mathcal{D}_{2}=\left[0,1\right]\cdot E+D$,
with $E$ the exceptional divisor of the blow up and $D$ the strict
transform of any curve passing with multiplicity one at the origin.
\item [$(iii)$]$X(Y_{3},\mathcal{D}_{3})$, where $\psi:Y_{3}\rightarrow Y_{1}$
is any projective birational morphism with $Y_{3}$ normal and $\mathcal{D}_{3}=\psi^{*}(\mathcal{D}_{1})$
is the total transform of $\mathcal{D}_{1}$, see \cite[lemma 2.1]{Li-S}.
\end{enumerate}
Amongst all the possible bases $Y$ for the combinatorial description
of a $\mathbb{T}$-variety, there is one of particular interest for
us. It corresponds to the definition of minimal p-divisor given in
\cite[Definition~8.7]{A-H}. A p-divisor $\mathcal{D}$ on a normal
semi-projective variety $Y$ is minimal if given a projective birational
morphism $\psi:Y\rightarrow Y'$ such that $\mathcal{D}=\psi^{*}(\mathcal{D}')$,
then we have $\psi$ is an isomorphism. Given an affine $\mathbb{T}$-variety
$X$, the semi-projective variety $Y$ where a minimal p-divisor describing
$X$ lives is called the Chow quotient for the $\mathbb{T}$-action.
In the previous example, the first two descriptions are minimal while
the last one is only minimal if $\psi$ is an isomorphism. 

We need to recall the construction of the Chow quotient of an affine
$\mathbb{T}$-variety $X$ in \cite{A-H}. For every $u\in M$ we
consider the set of semistable points
\[
X^{ss}(u)=\left\{ x\in X\mid\mbox{there exists }n\in\mathbb{Z}_{\geq0}\mbox{ and }f\in A_{nu}\mbox{ such that }f(x)\neq0\right\} .
\]
The set is a $\mathbb{T}$-invariant open subset of $X$ admitting
a good $\mathbb{T}$-quotient to
\[
Y_{u}=X^{ss}(u)/\!/\mathbb{T}=\mathrm{\operatorname{Proj}}\underset{n\in\mathbb{Z}_{\geq0}}{\bigoplus}A_{nu}.
\]
There exists a quasifan $\Lambda\in M_{\mathbb{Q}}$ generated by
a finite collection of (not necessarily strongly convex) cones $\lambda$
such that for any $u$ and $u'$ in the relative interior of the same
cone $\lambda\in\Lambda$, we obtain the same set of semi-stable points,
i.e., $X^{ss}(u)=X^{ss}(u')$. Thereby, we can index the open sets
semi-stable points $X^{ss}(u)$ using $\lambda$ so that $X^{ss}(u)=W_{\lambda}$
for any $u\in\mathrm{relint}(\lambda)$. Furthermore, if $\gamma$
is a face of $\lambda$, $W_{\lambda}$ is an open subset of $W_{\gamma}$.
Let $W=\underset{\lambda\in\Lambda}{\cap}W_{\lambda}$. The quotient
maps 
\[
q_{\lambda}:W_{\lambda}\rightarrow W_{\lambda}/\!/\mathbb{T}=\operatorname{Proj}\underset{n\in\mathbb{Z}_{\geq0}}{\bigoplus}A_{nu}
\]
for any $u$ in the relative interior of $\lambda$ form an inverse
system indexed by the cones in the fan $\Lambda$. We let $q:W\longrightarrow Y'=\underset{\leftarrow}{\lim}Y_{\lambda}$
be the inverse limit of this system. The semi-projective variety $Y$
is then obtained by taking the normalization of the closure of the
image of $W$ by $q$ in $Y'$. In the following commutative diagram
we summarize all the morphisms involved. We will refer to them by
these names in the text. The morphism $\overline{q}:W\rightarrow Y$
comes from the universal property of the normalization. 
\begin{lyxcode}
\begin{equation}
\xymatrix{W\ar[r]\ar[d]_{\overline{q}} & W_{\lambda}\ar[r]\ar[d]_{q_{\lambda}} & W_{\gamma}\ar[r]\ar[d]_{q_{\gamma}} & X\ar[dd]_{q_{0}}\\
Y\ar[drrr]_{\rho}\ar[r] & Y_{\lambda}\ar[drr]\ar[r] & Y_{\gamma}\ar[dr]\\
 &  &  & Y_{0}=\mathrm{Spec}(A_{0})
}
\label{eq:AH-diagram}
\end{equation}

\end{lyxcode}

\section{Smoothness criteria}

Our main result in this paper is a characterization of smooth $\mathbb{T}$-varieties
in terms of the A-H combinatorial data. This will follow as an application
of Luna's Slice Theorem \cite{L}. We will first recall this theorem
and the required notation. Throughout this section $G$ will be used
to denote a reductive algebraic group. The Slice Theorem describes
the local structure in the étale topology of algebraic varieties endowed
with a $G$-action. For details and proofs see \cite{L,Dr}. We will
follow the presentation in \cite[Appendix to Ch.~1\S D]{GIT}, see
also \cite{La2}. For details on étale morphisms and the étale topology,
see \cite{EGA,Stacks}. 

Let $X$, $X'$ be algebraic varieties, $f:X\rightarrow X'$ be a
morphism and $x\in X$. The morphism $f$ is called étale at $x$
if it is flat and unramified at $x$. We say that $f$ is étale if
it is étale at every point $x\in X$. For every $x\in X$, an étale
neighborhood $\psi$ of $x$ is an étale morphism $\psi:Z\rightarrow X$
from some algebraic variety $Z$ such that $x\in\psi(Z)$. We say
that two varieties $X$ and $X'$ are locally isomorphic in the étale
topology at the points $x\in X$ and $x'\in X'$ if they have a common
étale neighborhood, i.e., if there exists a variety $Z$ and a couple
of étale morphisms $\psi:Z\rightarrow X$ and $\psi':Z\rightarrow X'$
such that $x\in\psi(Z)$ and $x'\in\psi'(Z)$.

If $X$ is a $G$-variety, we say that the neighborhood $\psi$ is
$G$-invariant if and $Z$ is also a $G$-variety and $\psi$ is equivariant.
Furthermore, if both $X$ and $X'$ are $G$-varieties, we say that
$X$ and $X'$ are equivariantly locally isomorphic in the étale topology
at the points $x$ and $x'$ if they have a common $G$-invariant
étale neighborhood. 

For the Slice Theorem we need a stronger equivariant notion of étale
morphisms. Let $\phi:X\rightarrow X'$ be a $G$-equivariant morphism
of affine $G$-varieties. The morphism $\phi$ is called strongly
étale if
\begin{enumerate}
\item [(i)]$\phi$ is étale and the induced morphism $\phi_{/\!/G}:X/\!/G\rightarrow X'/\!/G$
is étale.
\item [(ii)] $\phi$ and the quotient morphism $\pi:X\rightarrow X/\!/G$
induce a $G$-isomorphism between $X$ and the fibered product $X'\times_{X'/\!/G}X/\!/G$.
\end{enumerate}
Strongly étale morphism play the role of ``local isomorphism'' in
the Slice Theorem. Let now $H$ be a closed subgroup of $G$ and $V$
an $H$-variety. The twisted product $G\star^{H}V$ is the algebraic
quotient of $G\times V$ with respect to the $H$-action given by
\[
H\times(G\times V)\rightarrow G\times V,\qquad(h,(g,v))\mapsto h\cdot(g,v)=(gh^{-1},hv)\,.
\]
For a smooth point $x\in X$, we denote by $T_{x}X$ the fiber at
$x$ of the tangent bundle of $X$. Furthermore, for a closed orbit
$G\cdot x$, we denote by $N_{x}$ the fiber at $x$ of the normal
bundle to the orbit $G\cdot x$. 
\begin{thm}
[Luna's Slice Theorem]\label{luna} Let $X$ be an smooth affine
$G$-variety. Let $x\in X$ be a point such that the orbit $G\cdot x$
is closed and fix a linearization of the trivial bundle on $X$ such
that $G\cdot x\subset X^{ss}$. Then, there exists an affine smooth
$G_{x}$-invariant locally closed subvariety $V_{x}$ of $X^{ss}$
(a slice to $G\cdot x$) such that we have the following commutative
diagram 
\begin{lyxcode}
\[
\xymatrix{G\star^{G_{x}}T_{x}V_{x}\ar[d] & G\star^{G_{x}}V_{x}\ar[r]\sp(0.55){st.\,\acute{e}t.}\ar[l]\sb(0.45){st.\,\acute{e}t.}\ar[d] & X^{ss}\ar[d]\\
N_{x}/\!/G_{x} & (G\star^{G_{x}}V_{x})/\!/G\ar[r]\sp(0.55){\acute{e}t.}\ar[l]\sb(0.5){\acute{e}t.} & X^{ss}/\!/G
}
\]

\end{lyxcode}
\end{thm}
Our main result will follow from the following proposition that translates
Luna's Slice Theorem into the language of the A-H presentation. Let
$X=X(Y,\mathcal{D})$. Recall that the morphism $\rho:Y\rightarrow Y_{0}$
in the A-H presentation is the natural morphism from the Chow quotient
$Y$ used in the presentation and the algebraic quotient $Y_{0}=X/\!/\mathbb{T}$. 
\begin{prop}
\label{prop:luna-AH} Let $X=X(Y,\mathcal{D})$ be an affine $\mathbb{T}$-variety
of dimension $n$. Assume that $Y$ is minimal. Then, $X$ is smooth
if and only if for every point $y$ of $Y_{0}$ there exists a neighborhood
$\mathcal{U}\subseteq Y_{0}$, an algebraic variety $Z$, a smooth
toric variety $X'$ of dimension $n$, and strongly étale morphisms
$a:Z\rightarrow X(\rho^{-1}(\mathcal{U}),\mathcal{D}|_{\rho^{-1}(\mathcal{U})})$
and $b:Z\rightarrow X'$ where and $\mathbb{T}$ acts in $X'$ via
a subtorus of the big torus action. In particular, every smooth point
in a $\mathbb{T}$-variety is $\mathbb{T}$-equivariantly locally
isomorphic in the étale topology to a point in the affine space endowed
with a linear $\mathbb{T}$-action.\end{prop}
\begin{proof}
Since $X$ is covered by the smooth $\mathbb{T}$-stable open sets
$X(\rho^{-1}(\mathcal{U}),\mathcal{D}|_{\rho^{-1}(\mathcal{U})})$,
we obtain that every point $x\in X$ is locally isomorphic in the
étale topology to a open set in $X'$. Since smoothness is preserved
by étale morphisms, we have that $X$ is smooth. This shows the ``if''
of the first assertion. 

Let now $q_{0}:X\rightarrow Y_{0}=X/\!/\mathbb{T}$ be the algebraic
quotient. Then for every $y\in Y_{0}$, $q_{0}^{-1}(y)$ contains
a unique closed orbit that we will denote $C_{y}$. Choose a linearization
of the trivial bundle such that $C_{y}\subset X^{ss}$. By the top
row in the commutative diagram in Theorem~\ref{luna}, for any point
$x\in C_{y}$ we have a closed smooth $\mathbb{T}_{x}$-stable subvariety
$V_{x}$ of $X$ and strongly étale morphisms 
\[
\xymatrix{X':=\mathbb{T}\star^{\mathbb{T}_{x}}T_{x}V_{x} & Z:=\mathbb{T}\star^{\mathbb{T}_{x}}V_{x}\ar[l]\sb(0.45){b}\ar[r]\sp(0.65){a} & X^{ss}.}
\]

Hence, we can take $\mathcal{U}$ as the image of $X^{ss}/\!/\mathbb{T}$
in $Y$ by $\overline{q}$ so that $X(\rho^{-1}(\mathcal{U}),\mathcal{D}|_{\rho^{-1}(\mathcal{U})})=X^{ss}$.
Since the $\mathbb{T}_{x}$-action on $\mathbb{T}$ and on $T_{x}V_{x}$
are linear, we have that $X'$ is a smooth toric variety where $\mathbb{T}$
acts as a subtorus of the big torus. This proves the ``only if''
part of the first assertion.

For the last statement, It is well known that every smooth affine
toric variety of dimension $n$ can be embedded as an open set in
$c:X'\hookrightarrow\mathbb{A}^{n}$ equivariantly with respect to
the big torus action. Since an open embedding is étale, the compositions
$c\circ b$ and $a$ are the required étale neighborhoods for every
$x\in X^{ss}\subseteq X$.\end{proof}
\begin{rem}
The statement that every smooth point in a $\mathbb{T}$-variety is
$\mathbb{T}$-equivariantly locally isomorphic in the étale topology
to a point in the affine space endowed with a linear $\mathbb{T}$-action
is generalization of the classical result that every smooth point
in a variety is locally isomorphic in the étale topology to a point
in the affine space. 
\end{rem}
For the proof of our main theorem, we need the following version of
\cite[Theorem~8.8]{A-H}.
\begin{lem}
\label{lem:AH-Th88}Let $X=X(Y,\mathcal{D})$ and $X'=X(Y',\mathcal{D}')$
be affine $\mathbb{T}$-varieties, where $\mathcal{D}$ and $\mathcal{D}'$
are minimal p-divisors on normal semi-projective varieties $Y$ and
$Y'$, respectively. If $\psi:X'\rightarrow X$ is a strongly étale
morphism, then, there exists an étale surjective morphism $\varphi:Y'\rightarrow Y$
such that $X'$ is equivariantly isomorphic to $X(Y',\varphi^{*}(\mathcal{D}))$. 

Conversely, if $\varphi:Y'\rightarrow Y$ is an étale surjective morphism
and $\mathcal{D}$ is a minimal p-divisor on $Y$, then $\mathcal{D}'=\varphi^{*}(\mathcal{D})$
is minimal and there is a strongly étale morphism. $\psi:X(Y',\mathcal{D}')\rightarrow X(Y,\mathcal{D})$.\end{lem}
\begin{proof}
For every $u\in M$, the existence of an étale morphism $(X')^{ss}(u)/\!/\mathbb{T}\rightarrow X^{ss}(u)/\!/\mathbb{T}$
is assured by the fact that $\psi$ is strongly étale. Since $\mathcal{D}$
and $\mathcal{D}'$ are minimal, we obtain, as in the proof of \cite[Theorem~8.8]{A-H},
an étale morphism $\varphi:Y'\rightarrow Y$. In this case, the morphism
$\kappa$ therein can be taken as the identity, see \cite[p.~600~l.1]{A-H}.
Hence, \cite[Theorem~8.8]{A-H} yields $X'=X(Y',\varphi^{*}(\mathcal{D}))$
equivariantly.

Conversely, in this case we have a morphism of p-divisors $(\varphi,\mathrm{id},1):\mathcal{D}'\rightarrow\mathcal{D}$.
Hence, by \cite[Proposition~8.6]{A-H} we obtain the existence of
a morphism $\psi:X(Y',\mathcal{D}')\rightarrow X(Y,\mathcal{D})$.
By the diagram in equation (\ref{eq:AH-diagram}), the A-H presentation
is local relative to the algebraic quotient morphism $q_{0}:X\rightarrow Y_{0}=X/\!/\mathbb{T}$.
Hence, we can assume we have a morphism $X=X^{ss}(u)=W_{\lambda}\rightarrow Y=Y_{\lambda}=W_{\lambda}/\!/\mathbb{T}$.
Let now $W=X\times_{Y}Y'$ be the fiber product. By definition, the
morphism $W\rightarrow X$ is strongly étale and the Chow quotient
of $W$ is $Y'$. By \cite[Theorem~8.8]{A-H} and its proof, since
we have morphism $\varphi:Y'\rightarrow Y$ between the Chow quotients,
we obtain again that $\kappa$ therein can be taken as the identity
and so $W=X(Y',\varphi^{*}(\mathcal{D}))$.
\end{proof}
Let now $\mathcal{D}$ and $\mathcal{D}'$ be p-divisors on normal
semi-projective varieties $Y$ and $Y'$ respectively. We say that
$(Y,\mathcal{D})$ is locally isomorphic in the étale topology to
$(Y',\mathcal{D}')$ if for every couple of points $y\in Y$ and $y'\in Y'$
we have a variety $V$ and étale neighborhoods $\psi:V\rightarrow Y$
and $\psi':V\rightarrow Y'$ such that $\psi^{*}(\mathcal{D})=(\psi')^{*}(\mathcal{D}')$.
We can now prove our main theorem.
\begin{thm}
\label{Thm: smoothness} Let $X=X(Y,\mathcal{D})$ be an affine $\mathbb{T}$-variety
of dimension $n$. Assume that $(Y,\mathcal{D})$ is minimal. Then,
$X$ is smooth if and only if for every point $y\in Y_{0}$ there
exists a neighborhood $\mathcal{U}\subseteq Y_{0}$, a linear $\mathbb{T}$-action
on $\mathbb{A}^{n}$ given by the combinatorial data $(Y',\mathcal{D}')$
with $\mathcal{D}'$ minimal, an algebraic variety $V$, and étale
morphisms $\alpha:V\rightarrow Y$ and $\beta:V\rightarrow Y'$ such
that $\rho^{-1}(\mathcal{U})\subseteq\alpha(V)$ and $\alpha^{*}(\mathcal{D})=\beta^{*}(\mathcal{D}')$.
In particular, the combinatorial data $(Y,\mathcal{D})$ is locally
isomorphic in the étale topology to the combinatorial data of the
affine space endowed with a linear $\mathbb{T}$-action. \end{thm}
\begin{proof}
Since $X$ is covered by the smooth $\mathbb{T}$-stable open sets
$X(\rho^{-1}(\mathcal{U}),\mathcal{D}|_{\rho^{-1}(\mathcal{U})})$.
We obtain that every point $x\in X$ is locally isomorphic in the
étale topology to a open set in $\mathbb{A}^{n}$ endowed with a linear
$\mathbb{T}$-action. Hence $X$ is smooth. This shows the ``if''
of the first assertion. 

Assume now that $X$ is smooth. By the diagram in equation (\ref{eq:AH-diagram}),
the A-H presentation is local relative to the algebraic quotient morphism
$q_{0}:X\rightarrow Y_{0}=X/\!/\mathbb{T}$. Hence, we can assume
we have a morphism $X=X^{ss}(u)=W_{\lambda}\rightarrow Y=Y_{\lambda}=W_{\lambda}/\!/\mathbb{T}$.
By Proposition~\ref{prop:luna-AH}, there exists an algebraic variety
$Z$, a smooth toric variety $X'$ of dimension $n$, and strongly
étale morphisms $a:Z\rightarrow X(\rho^{-1}(\mathcal{U}),\mathcal{D}|_{\rho^{-1}(\mathcal{U})})$
and $b:Z\rightarrow X'$ where and $\mathbb{T}$ acts in $X'$ via
a subtorus of the big torus action. By the diagram in Theorem~\ref{luna}
we obtain the following diagram of étale maps

\[
\xymatrix{X':=\mathbb{T}\star^{\mathbb{T}_{x}}T_{x}V_{x}\ar[d] & Z:=\mathbb{T}\star^{\mathbb{T}_{x}}V_{x}\ar[l]\sb(0.45){b}\ar[r]\sp(0.55){a}\ar[d] & X=X^{ss}\ar[d]\\
Y'=N_{x}/\!/\mathbb{T}_{x} & V:=(\mathbb{T}\star^{\mathbb{T}_{x}}V_{x})/\!/\mathbb{T}\ar[r]\sp(0.58){\alpha}\ar[l]\sb(0.55){\beta} & Y=X^{ss}/\!/\mathbb{T}
}
\]
The result now follows by applying twice Lemma~\ref{lem:AH-Th88}
to the right and left square in the above diagram and by embedding
the smooth affine toric variety $X'$ linearly on the affine space
$\mathbb{A}^{n}$ as in the proof of Proposition~\ref{prop:luna-AH}.\end{proof}
\begin{rem}
By Theorem~\ref{Thm: smoothness}, local models for smooth $\mathbb{T}$-varieties
can be obtained via the downgrading procedure described in \cite[section 11]{A-H}.
In Section~\ref{sec:threefolds}, we will use this procedure to compute
the local models for $\mathbb{G}_{m}$-threefolds. Here, we now give
as a simple corollary the well known criterion for the smoothness
of a complexity one $\mathbb{T}$-variety \cite{Li-S}.\end{rem}
\begin{cor}
\label{cor:cp1}Let $X=X(Y,\mathcal{D})$ be an affine $\mathbb{T}$-variety
of complexity one, where $Y$ is a smooth curve and $\mathcal{D}=\sum\Delta_{i}\cdot z_{i}$
is a p-divisor on $Y$. Then $X$ is smooth if and only if 
\begin{enumerate}
\item [$(i)$] $Y$ is affine and the cone spanned in $N_{\mathbb{Q}}\times\mathbb{Q}$
by $(0,\operatorname{tail}(\mathcal{D}))$ and $(1,\Delta_{i})$ is
smooth for all $i$.
\item [$(ii)$] $Y=\mathbb{P}^{1}$ and $X$ is the affine space endowed
with a linear $\mathbb{T}$-action.
\end{enumerate}

Furthermore, if $X$ is rational then there is a Zariski covering
of $X$ by open sets isomorphic to open sets in the affine space endowed
with linear $\mathbb{T}$-actions. 

\end{cor}
\begin{rem}
By the results from the second author in \cite{P}, even in the case
where $X$ is rational, there is not always a Zariski covering of
$X$ by open sets isomorphic to open sets in the affine space endowed
with linear $\mathbb{T}$-actions in complexity higher than one.
\end{rem}

\section{Smooth threefolds with complexity two torus actions\label{sec:threefolds}}

To be able to effectively apply Theorem \ref{Thm: smoothness} as
a smoothness criterion, it is useful to compute all possible A-H presentations
for the affine space $\mathbb{A}^{n}$ endowed with a linear torus
action. As we show in the previous section, a smoothness criterion
in complexity $1$ is well known and is easily reproved as a corollary
of with our criterion. In this section we compute all A-H presentations
for the  affine space $\mathbb{A}^{3}$ with a linear torus action
of complexity $2$. This yields an effective smoothness criterion
for threefolds endowed with a $\mathbb{T}$-action of complexity $2$.

To be able to perform the computations, we need to be able to find
a set of generators of the rational quotient of the action. This forces
us to take suitable $\mathbb{T}$-invariant cyclic cover along a coordinate
axis. Such cyclic covers are well understood for the A-H presentation
following the work of the second author, see \cite{P1}. Recall that
a cyclic cover of $\mathbb{A}^{n}$ along a coordinate axis is isomorphic
to $\mathbb{A}^{n}$.

In the following proposition, we denote a section of the matrix $F=\begin{pmatrix}a & b & \pm c\end{pmatrix}^{t}$
by $s=\begin{pmatrix}\alpha & \beta & \gamma\end{pmatrix}$. We also
denote $\mathrm{gcd}(i,j)$ by $\rho(i,j)$ the and by $\delta=\mathrm{gcd}(\frac{a}{\rho(a,c)},\frac{b}{\rho(b,c)})$.
\begin{prop}
\label{classification A3  } Let $X=\mathbb{A}^{3}$ endowed with
an effective $\mathbb{G}_{m}$-action. Then, after a suitable $\mathbb{T}$-invariant
cyclic cover along a coordinate axis, we have $X\simeq\mathbb{A}^{3}\simeq X(Y,\mathcal{D})$
with weight matrix $F$, where $F$, $Y$ and $\mathcal{D}$ are given
in the following list: 
\begin{enumerate}
\item $F=\begin{pmatrix}a & b & -c\end{pmatrix}^{t}$ with $a$, $b$ and
$c$ positive integers; $Y$ is isomorphic to a weighted blow-up $\pi:\tilde{\mathbb{A}}^{2}\rightarrow\mathbb{A}^{2}$
of $\mathbb{A}^{2}$ centered at the origin having an irreducible
exceptional divisor $E$; and $\mathcal{D}$ is given by 
\[
\mathcal{D}=\left\{ \frac{\alpha\rho(a,c)}{c}\right\} \otimes D_{1}+\left\{ \frac{\beta\rho(b,c)}{c}\right\} \otimes D_{2}+\left[\frac{\gamma}{\delta},\frac{\gamma}{\delta}+\frac{1}{\delta c}\right]\otimes E,
\]
 with $D_{1}$, $D_{2}$ the strict transforms of the coordinate axes
in $\mathbb{A}^{2}$.
\item $F=\begin{pmatrix}a & b & c\end{pmatrix}^{t}$ with $a$, $b$ and
$c$ positive integers; $Y$ is isomorphic to the weighted projective
space $\mathbb{P}(a,b,c)$ having a smooth standard chart $U_{3}=\left\{ x_{3}\neq0\right\} $;
and $\mathcal{D}$ is given by 
\[
\mathcal{D}=\left[\frac{\alpha\rho(a,c)}{c};+\infty\right[\otimes D_{1}+\left[\frac{\beta\rho(b,c)}{c};+\infty\right[\otimes D_{2}+\left[\frac{\gamma}{\delta};+\infty\right[\otimes D_{3},
\]
 with $D_{1}$, $D_{2}$ and $D_{3}$ the coordinate axes in $\mathbb{P}(a,b,c)$.
\item $F=\begin{pmatrix}0 & b & c\end{pmatrix}^{t}$ with $b$ and $c$
positive integers; $Y$ is isomorphic to $\mathbb{P}^{1}\times\mathbb{A}^{1}$;
and $\mathcal{D}$ is given by 
\[
\mathcal{D}=\left[\frac{\beta\rho(b,c)}{c};+\infty\right[\otimes D_{2}+\left[-\frac{\beta\rho(b,c)}{c};+\infty\right[\otimes D_{3},
\]
with $D_{2}=\left\{ 0\right\} \times\mathbb{A}^{1}$ and $D_{3}=\left\{ \infty\right\} \times\mathbb{A}^{1}$.
\item $F=\begin{pmatrix}0 & b & -c\end{pmatrix}^{t}$ with $b$ and $c$
positive integers; $Y$ is isomorphic to $\mathbb{A}^{2}$; and $\mathcal{D}$
is given by 
\[
\mathcal{D}=\left[\frac{\gamma\rho(b,c)}{b},\frac{\beta\rho(b,c)}{c}\right]\otimes D_{2},
\]
with $D_{2}$ a coordinate axis in $\mathbb{A}^{2}$.
\item $F=\begin{pmatrix}0 & 0 & 1\end{pmatrix}^{t}$; $Y$ is isomorphic
to $\mathbb{A}^{2}$; and $\mathcal{D}=0$ with tail cone $\left[0,+\infty\right[$.
\end{enumerate}
\end{prop}
\begin{proof}
Let $\mathbb{A}^{3}$ be endowed with a linear action of $\mathbb{G}_{m}$,
we will give the ingredients to compute the A-H presentation by the
downgrading method described in \cite[section 11]{A-H}. Therein,
we need to from the following exact sequence
\[
\xymatrix{0\ar[r] & \mathbb{Z}\ar[r]_{F} & \mathbb{Z}^{3}\ar[r]_{P}\ar@/_{1pc}/[l]_{s} & \mathbb{Z}^{2}\ar[r] & 0}
,
\]

In general, in cases (1) and (2) it is impossible to compute the cokernel
matrix $P$. Nevertheless, the choices made in this two cases of the
proposition up to cyclic covering, allows us to compute $P$. Indeed,
a direct verification shows that $P$ can be chosen, respectively,
as follows:
\begin{enumerate}
\item \begin{multicols}{2}$P=\left(\begin{array}{ccc}
\frac{c}{\rho(a,c)} & 0 & \frac{a}{\rho(a,c)}\\
0 & \frac{c}{\rho(b,c)} & \frac{b}{\rho(b,c)}
\end{array}\right)$, 
\item $P=\left(\begin{array}{ccc}
\frac{c}{\rho(a,c)} & 0 & -\frac{a}{\rho(a,c)}\\
0 & \frac{c}{\rho(b,c)} & -\frac{b}{\rho(b,c)}
\end{array}\right)$,
\item $P=\left(\begin{array}{ccc}
1 & 0 & 0\\
0 & \frac{c}{\rho(b,c)} & -\frac{b}{\rho(b,c)}
\end{array}\right)$, 
\item $P=\left(\begin{array}{ccc}
1 & 0 & 0\\
0 & \frac{c}{\rho(b,c)} & \frac{b}{\rho(b,c)}
\end{array}\right)$, 
\item $P=\left(\begin{array}{ccc}
1 & 0 & 0\\
0 & 1 & 0
\end{array}\right)$. \end{multicols}
\end{enumerate}

In the last three cases, it is easy to verify that $P$ is indeed
the cokernel of $F$. In the first case, the assumption on the base
$Y$ allows to conclude that the algebraic quotient $X/\!/\mathbb{G}_{m}$
is isomorphic to $\mathbb{A}^{2}$. Hence, it has only two generators
that are given by the two row of $P$. Similarly, in the second case,
the smoothness assumption of the chart $U_{3}$ in $\mathbb{P}(a,b,c)$
yields the same conclusion.

\end{proof}
In the remaining of this section, we provide several illustrative
examples showing the different behaviors of the A-H description in
the case of complexity $2$ threefolds with respect to our smoothness
criterion.
\begin{example}
[A smooth $\mathbb{T}$-variety over an open set in an abelian variety]Let
$E_{1}=\{h_{1}(u_{1},v_{1})=0\}\subset\mathbb{A}^{2}=\mathrm{Spec}(\mathrm{k}[u_{1},v_{1}])$
and $E_{2}=\{h_{2}(u_{2},v_{2})=0\}\subset\mathbb{A}^{2}=\mathrm{Spec}(\mathrm{k}[u_{2},v_{2}])$
be two planar affine smooth elliptic curves passing with multiplicity
one through the origin. Let $Y$ be the blow up of $E_{1}\times E_{2}$
at the origin of $\mathbb{A}^{4}=\mathbb{A}^{2}\times\mathbb{A}^{2}$.
Let $X=X(Y,\mathcal{D})$, where
\[
\mathcal{D}=\left\{ \frac{1}{3}\right\} \otimes D+\left[0,\frac{1}{3}\right]\otimes E,
\]
where $D$ is the strict transform of $\left\{ u_{1}=0\right\} \times E_{2}\subseteq E_{1}\times E_{2}$
in $Y$ and $E$ is the exceptional divisor of the blowup. We will
show that $X$ is smooth. By Theorem~\ref{Thm: smoothness}, we only
need to show that for every $y\in Y_{0}=X/\!/\mathbb{T}\simeq E_{1}\times E_{2}$
there is an étale neighborhood $\mathcal{U}$ such that $X\left(\rho^{-1}(\mathcal{U}),\mathcal{D}|_{\rho^{-1}(\mathcal{U})}\right)$
is $\mathbb{T}$-equivariantly isomorphic to an étale neighborhood
of $\mathbb{A}^{3}$ endowed with a linear $\mathbb{G}_{m}$-action.
Hence, we only need to show that, étale locally over $y\in Y_{0}$,
the divisor $\mathcal{D}$ appears in the list in Proposition~\ref{classification A3  }.
Indeed, it corresponds to case (1) with $F=\begin{pmatrix}1 & 1 & -3\end{pmatrix}^{t}$
taking $s=\begin{pmatrix}0 & 1 & 0\end{pmatrix}$.

On the other hand, with the method in \cite[section 11]{A-H}, we
can verify that $X=X(Y,\mathcal{D})$ is isomorphic to the closed
$\mathbb{G}_{m}$-stable subvariety of $\mathbb{A}^{5}=\mathrm{Spec}\left(\mathrm{k}[x_{1},y_{1},x_{2},y_{2},z]\right)$
with weight matrix $\begin{pmatrix}1 & 3 & 3 & 3 & -3\end{pmatrix}^{t}$
given by the equations
\[
\left\{ \tfrac{1}{z}\cdot h_{1}(x_{1}^{3}z,y_{1}z)=0\,;\,\tfrac{1}{z}\cdot h_{2}(x_{2}z,y_{2}z)=0\right\} \subset\mathbb{A}^{5}.
\]
By the Jacobian criterion, we can check that $X$ is indeed smooth. 
\end{example}
 
\begin{example}
[A smooth $\mathbb{T}$-variety with non-rational support divisor]Let
$E=\{h(u,v)=u^{2}-v(v-\alpha)(v-\beta)=0\}\subset\mathbb{A}^{2}=\mathrm{Spec}(\mathrm{k}[u,v])$
be a planar affine smooth elliptic curve. Let $X=X(Y,\mathcal{D})$,
where $Y$ is the blowup of $\mathbb{A}^{2}$ at the origin and 
\[
\mathcal{D}=\left\{ \frac{1}{2}\right\} \otimes D_{1}+\left\{ -\frac{1}{3}\right\} \otimes D_{2}+\left[0,\frac{1}{6}\right]\otimes E,
\]
where $D_{1}$ is the strict transform of the curve $\{h(u,v)=0\}$,
$D_{2}$ is the strict transform of $\{u=0\}$ and $E$ is the exceptional
divisor of the blowup. Again, by Theorem~\ref{Thm: smoothness},
we only need to show that, étale locally over $y\in Y_{0}=X/\!/\mathbb{G}_{m}\simeq\mathbb{A}^{2}$,
the divisor $\mathcal{D}$ appears in the list in Proposition~\ref{classification A3  }.
In an étale neighborhood of the origin in $\mathbb{A}^{2}$ it corresponds
to case (1) with $F=\begin{pmatrix}2 & 3 & -6\end{pmatrix}^{t}$ taking
$s=\begin{pmatrix}-1 & 1 & 0\end{pmatrix}$. Furthermore, the same
local model works in an étale neighborhood of every point different
from $(0,\alpha)$ and $(0,\beta)$. In an étale neighborhood $\mathcal{U}$
of these two points, the divisor $\mathcal{D}$ is given by $\mathcal{D}|_{\mathcal{U}}=\left\{ \frac{1}{2}\right\} \otimes D_{1}+\left\{ -\frac{1}{3}\right\} \otimes D_{2}$.
Such a p-divisor corresponds to a $\mathbb{T}$-invariant open set
of case (2) with $F=\begin{pmatrix}2 & 3 & 6\end{pmatrix}^{t}$ taking
$s=\begin{pmatrix}-1 & 1 & 0\end{pmatrix}$, see \cite[Proposition 3.4]{A-H-S}.

On the other hand, we can verify that $X=X(Y,\mathcal{D})$ is isomorphic
to the closed $\mathbb{G}_{m}$-stable subvariety of $\mathbb{A}^{4}=\mathrm{Spec}(\mathrm{k}[x,y,z,t])$
with weight matrix $\begin{pmatrix}2 & 6 & -6 & 3\end{pmatrix}^{t}$
given by the equation
\[
\left\{ \tfrac{1}{z}\cdot h(x^{3}z,yz)=t^{2}\right\} \subset\mathbb{A}^{5}.
\]
By the Jacobian criterion, we can check that $X$ is indeed smooth. 
\end{example}
 
\begin{example}
[A singular $\mathbb{T}$-variety with smooth combinatorial data]
Let $C=\{h(u,v)=u+v(1-v)^{2}=0\}\simeq\mathbb{A}^{1}\subset\mathbb{A}^{2}=\mathrm{Spec}(\mathrm{k}[u,v])$.
Let $X=X(Y,\mathcal{D})$, where $Y$ is the blowup of $\mathbb{A}^{2}$
at the origin and 
\[
\mathcal{D}=\left\{ \frac{1}{2}\right\} \otimes D_{1}+\left\{ -\frac{1}{3}\right\} \otimes D_{2}+\left[0,\frac{1}{6}\right]\otimes E,
\]
where $D_{1}$ is the strict transform of the curve $C$, $D_{2}$
is the strict transform of $\{u=0\}$ and $E$ is the exceptional
divisor of the blowup. By Theorem~\ref{Thm: smoothness}, $X$ is
smooth if and only if, étale locally over $y\in Y_{0}=X/\!/\mathbb{G}_{m}\simeq\mathbb{A}^{2}$,
the divisor $\mathcal{D}$ appears in the list in Proposition~\ref{classification A3  }.
In an étale neighborhood of of every point different from $(0,1)$
it corresponds to case (1) with $F=\begin{pmatrix}2 & 3 & -6\end{pmatrix}^{t}$
taking $s=\begin{pmatrix}-1 & 1 & 0\end{pmatrix}$. Nevertheless,
$D_{1}$ and $D_{2}$ intersect non-normally on the preimage of $(0,1)$.
Since all toric divisors intersect normally, such point does not admit
an étale neighborhood $\mathcal{U}$ such that $X\left(\rho^{-1}(\mathcal{U}),\mathcal{D}|_{\rho^{-1}(\mathcal{U})}\right)$
is equivariantly isomorphic to an étale open set on a toric variety. 

On the other hand, we can verify that $X=X(Y,\mathcal{D})$ is isomorphic
to the $\mathbb{G}_{m}$-stable subvariety of $\mathbb{A}^{4}=\mathrm{Spec}(\mathrm{k}[x,y,z,t])$
with weight matrix $\begin{pmatrix}2 & 6 & -6 & 3\end{pmatrix}^{t}$
given by the equation
\[
\{x^{3}+y(1-yz)^{2}=t^{2}\}\subset\mathbb{A}^{4}.
\]
By the Jacobian criterion, we can check that the point $(0,1,1,0)$
is singular.
\end{example}
 
\begin{example}
[A singular $\mathbb{T}$-variety with irreducible support] Let $X=X(Y,\mathcal{D})$,
where $Y$ is the blowup of $\mathbb{A}^{2}$ at the origin and 
\[
\mathcal{D}=\left[-p,0\right]\otimes E,
\]
where $E$ is the exceptional divisor of the blowup and $p$ is an
integer strictly greater than one. By Theorem~\ref{Thm: smoothness},
$X$ cannot be smooth since an exceptional divisor only appears on
case (1) in the list in Proposition~\ref{classification A3  } and
therein, the width of the coefficient polytope is at most 1. On the
other hand, we can verify that $X$ is equivariantly isomorphic to
the quotient of $\mathbb{A}^{3}=\mathrm{Spec}(\mathrm{k}[x,y,z])$
by the finite cyclic group $\mu_{p}$ of the $p$-th roots of the
unit acting via $\epsilon\cdot(x,y,z)=(\epsilon x,\epsilon y,z)$.
Since such action is not a pseudo-reflection, the quotient is singular
\cite{Ch}. The $\mathbb{G}_{m}$-action is given in $\mathbb{A}^{3}$
by weight matrix $\begin{pmatrix}1 & 1 & -1\end{pmatrix}^{t}$. \end{example}

\end{document}